\documentclass{article}
\usepackage{amssymb}
\usepackage{amsmath}
\usepackage{amsthm}
\usepackage{verbatim}
\usepackage{color}
\usepackage{url}
 \usepackage[all]{xy}
\usepackage{fancyhdr}

\newtheorem{thm}{Theorem}

\newtheorem{lem}{Lemma}
\newtheorem{cor}{Corollary}

\newcommand{\sabs}[1]{\left|#1\right|}
\newcommand{\sparen}[1]{\left(#1\right)}
\newcommand{\norm}[1]{\sabs{\sabs{#1}}}

\numberwithin{equation}{section}
\title{Stability Analysis in Magnetic Resonance Elastography II\\\vskip 0.8cm}
\author{ Heiko Gimperlein\thanks{Maxwell Institute for Mathematical Sciences and Department of Mathematics, Heriot--Watt University, Edinburgh, EH14 4AS, United Kingdom, email: h.gimperlein@hw.ac.uk}  \thanks{Institute for Mathematics, University of Paderborn, Warburger Str.~100, 33098 Paderborn, Germany} \and Alden Waters\thanks{Department of Mathematics, University College London, Gower Street, London, WC1E 6BT, United Kingdom, email: alden.waters@ucl.ac.uk }}
\date{}

\begin{document}

\maketitle

\begin{abstract}
We consider the inverse problem of finding unknown elastic parameters  from internal measurements of displacement fields for tissues. In the sequel to \cite{SAMRE}, we use pseudodifferential methods for the problem of recovering the shear modulus for Stokes systems from internal data. We prove stability estimates in $d=2,3$ with reduced regularity on the estimates and show that the presence of a finite dimensional kernel can be removed. This implies the convergence of the Landweber numerical iteration scheme. We also show that these hypotheses are natural for experimental use in constructing shear modulus distributions. 
\end{abstract}

\bigskip

\noindent {\footnotesize Mathematics Subject Classification
(MSC2000): 35B30.}


\section{Introduction}

This article uses pseudodifferential methods to sharpen recent stability estimates for an inverse problem of Magnetic Resonance Elastography (MRE), with short proofs. The new proofs allow to analyse practical numerical reconstruction methods. 

In Magnetic Resonance Elastography, internal measurements of time-harmonic displacement fields offer the possibility of a highly resolved reconstruction of shear modulus distributions. It is motivated by the detection of cancerous anomalies in their early stages \cite{ammaribio}. See \cite{Bal, Bal2, Kuchment, Kuchment2, Montalto, Nachman, Nakamura} for relevant related works. 

The reduced regularity assumptions of the stability estimates in this article prove relevant for numerical reconstruction schemes. We show how they directly relate to classical results for overdetermined elliptic boundary problems and their analysis using pseudodifferential operators. The analysis is based on a Stokes system as in \cite{SAMRE}. 

Let $\Omega$ denote a simply-connected bounded domain in $\mathbb{R}^d$ where $d=2, 3$ with $\mathcal{C}^\infty$-boundary $\partial\Omega$. We consider the following boundary value problem for the elasticity equations
\begin{align}\label{model}
\left\{
\begin{array}{lr}
\nabla(\lambda(x)\nabla\cdot u_{\lambda})+\omega^2 u_{\lambda}(x)+2\nabla\cdot\mu(x)\nabla^s u_{\lambda}(x)=0 \,\quad \mathrm{in}\,\, \Omega ,\\ \\
u_{\lambda}(x)=F(x) \,\quad \mathrm{on} \,\quad \partial\Omega \ ,\\
\end{array}
\right.
\end{align}
where 
\begin{align*}
2\nabla^s u_{\lambda}=\nabla u_{\lambda}+(\nabla u_{\lambda})^t
\end{align*}
and $\nabla u_{\lambda}$ is the matrix $(\partial_ju_{\lambda, i})_{i,j=1}^d$, with $u_{\lambda, i}$ the $i$-th component of $u_\lambda$.
The Lam\'{e} coefficients
(respectively, the shear and the compressional modulus) $\lambda, \mu \in \mathcal{C}^1(\bar{\Omega})$ satisfy
\begin{align}
\label{assumpmin}
& \lambda\geq \lambda_{min}= \min\{\lambda(x): x\in \bar{\Omega}\} >0, \\
& \mu\geq \mu_{min}= \min\{\mu(x): x\in \bar{\Omega}\} >0.\label{assumpmin2}
\end{align}

If $F\in H^{1/2}(\partial\Omega)$ and (\ref{assumpmin}) and (\ref{assumpmin2}) are satisfied, there exists a unique solution $u_{\lambda} \in H^1(\Omega)^d$ to \eqref{model} even for $\lambda, \mu \in L^\infty(\Omega)$. In particular, $\nabla^su_{\lambda}\in L^2(\Omega)^d$.
Higher regularity $\nabla^su_{\lambda}\in H^4(\Omega)^d$ holds under the additional assumption that $\lambda, \mu \in \mathcal{C}^4(\bar{\Omega})$, $F\in H^{9/2}(\partial\Omega)^d$.\\

In \cite{SAMRE} it was shown that for $2\mu_{\max}<3\lambda_{\min}$ the solution $u_{\lambda}(x)$ is approximated in $H^1(\Omega)^d$ norm up to $\mathcal{O}(\lambda^{-1/2})$ by the solution to the Stokes problem
\begin{align}\label{stokesdata}
\left\{
\begin{array}{lr}
\omega^2 u(x)+2\nabla\cdot\mu(x)\nabla^s u(x)+\nabla p(x)=0 \,\quad \mathrm{in}\,\, \Omega ,\\ \\
\nabla\cdot u(x)=0 \,\quad \mathrm{in} \,\quad \Omega , \\\\
u(x)=F(x) \,\quad \mathrm{on} \,\quad \partial\Omega ,\\ \\
\int\limits_{\Omega}p(x)\,dx=0.
\end{array}
\right.
\end{align}

If we examine the solutions in 2 and 3 dimensions, we find we can reconstruct a displacement of the shear modulus $\mu$ in a stable way from the difference of the solutions. Hybrid modalities involve exciting the system with more than one wave or modality. In $3d$ it was already shown in \cite{SAMRE} using elliptic regularity theory that hybrid modalities are necessary for reconstruction of the shear modulus. This necessity is in contrast to dimension 2. The goal of this article is to show how pseudodifferential operators and the theory of elliptic boundary problems leads to simple proofs of stronger stability estimates for the inverse problem. The estimates reduce the necessary order of regularity and allow the analysis of practical numerical reconstruction methods. 

In Section \ref{fredsection} we introduce basic notions of elliptic boundary problems and the Shapiro-Lopatinskii ellipticity condition. A basic Fredholm theorem is recalled in Theorem \ref{fredholm2}. In sections \ref{stab2d} and \ref{stab3d} we apply Theorem \ref{fredholm2} to conclude basic stability estimates (Theorems \ref{conditional} and \ref{conditional2}) in the presence of a finite dimensional kernel. Refined stability estimates in $L^2$, based on stronger assumptions, are the content of Section 5. We also show that locally it is possible to remove the presence of the finite dimensional kernel under certain hypotheses on the system, Corollaries \ref{cor2d} and \ref{cor3d}. These results imply convergence of the numerical Landweber iteration scheme, which is discussed in Section \ref{secfinal}. We also show that the assumptions which are made on the symbols are natural for numerical experiments in Section \ref{example}. \\

\noindent \emph{Notation:} 
In this paper we use the Einstein summation convention. For two matrices $A$ and $B$, the inner product is denoted by
 \begin{align*}
A:B=a_{ij}b_{ji},
\end{align*}
and we write $|A|^2 = A:A$.
For vector--valued functions
\begin{align*}
f(x)=(f_1(x),f_2(x),\ldots,f_d(x)):\Omega\rightarrow \mathbb{R}^d\ ,
\end{align*}
the Hilbert space $H_{m}(\Omega)^d$, $m\in \mathbb{N}$ is defined as the completion of the space $\mathcal{C}_c^{\infty}(\Omega)^d$ with respect to the norm
\begin{align*}
\|f\|^2_m = \|f\|^2_{m,\Omega} = \sum\limits_{|i|=1}^m \int\limits_{\Omega}\sparen{|\nabla^i f(x)|^2+|f(x)|^2}\,dx,
\end{align*}
where we write $\nabla^i= \partial^{i_1} \ldots \partial^{i_d}$ for $i=(i_1,\ldots, i_d)$ for the higher-order derivative. 

The outward unit normal to the boundary $\partial\Omega$ is denoted by $n$.
If $\mu\in \mathcal{C}^1(\Omega)$, we define the conormal derivative
\begin{align*}
2\frac{\partial f}{\partial \nu}=\mu(x)\sparen{\nabla f+(\nabla f)^t}n\ .
\end{align*}

We say that a linear operator $T: C^{\infty}_c(\mathbb{R}^d) \to C^{\infty}(\mathbb{R}^d)$ is of order $\leq \alpha$ if 
\begin{align*}
\norm{Tu}_{m}\leq C\norm{u}_{m+\alpha}\ , \qquad u\in C^{\infty}(\mathbb{R}^d) \ ,
\end{align*}
for every $m\in \mathbb{R}$. We say $\mathrm{ord}(T) = \alpha$ if $T$ is of order $\leq \alpha$, but not of order $\leq \alpha'$ for $\alpha'<\alpha$.

\section{Reduction of the Regularity for Stability Estimates}\label{fredsection}

This section recalls some aspects of the elliptic theory for boundary problems in low--order Sobolev spaces. These will be applied to obtain regularity estimates for the inverse problem which are suitable for numerical calculations. 

We outline a direct approach to Fredholm estimates for a class of boundary problems relevant to the reconstruction problem. Our presentation will follow the treatment in Eskin \cite{eskin}, Ch.~52, to account for the low--order Sobolev spaces. \\

After straightening out the boundary, it is sufficient to consider the boundary problem in the half-space
\begin{align*}
\mathbb{R}_+^n=\{(x',x_n): x_n>0, \,\,x'\in \mathbb{R}^{n-1}\}.
\end{align*}

 We consider $A(x,D)$ as an $r\times r$ system of differential operators of order $m$, with principal part $A_0(x,D)$. We say than an operator $A(x,D)$ is elliptic if the determinant of the principal symbol $\mathrm{det}A_0(x,\xi)$ defines the symbol of a scalar elliptic operator of order $mr$. In other words, we have $\mathrm{det}A_0(x,\xi)\neq 0$ for all $\xi\neq 0$.  We consider the initial boundary value problem
\begin{align*}
&A(x,D)u=f\qquad \mathrm{in} \quad \Omega \\&
B_j(x',D)u|_{\partial\Omega}=g_j \qquad 1\leq j\leq m_{+}=m/2
\end{align*}
We let $B(x',D)$ be the $m_+\times r$ matrix with rows $B_j$. The system of boundary conditions can then be written in the following form
\begin{align*}
B(x',D)u|_{\partial\Omega}=g,
\end{align*}
with $g=(g_1, . . ,g_{m_+})$. We can associate a family $A_0(x',\xi'_{x'},0,D_n)$ and $B_0(x',\xi_{x'}',0,D_n)$ which are differential operators on $\mathbb{R}_+^1$ depending on $(x',\xi_{x'}')\in T_0^*(\partial\Omega)$. 

The Shapiro--Lopatinskii condition can be stated as the unique solvability of the differential equation in $L^2(\mathbb{R}_+)$. In order to verify this condition, we translate it into a non-vanishing condition on a determinant associated to the symbols of the differential operators.  

If $A_0(x,0,+1)=1$, we obtain a factorization $A_0(x,\xi',\xi_n)=A_+(x,\xi',\xi_n)A_-(x,\xi',\xi_n)$, with
\begin{align*}
A_\pm(x,\xi',\xi_n)=\prod_{j=1}^{m_+}(\xi_n-\sigma_j^\pm(x,\xi')).
\end{align*}

For any $N\geq s$ an integer, we construct operators with symbols $R_p^+(x,\xi',\xi_n)$ and $\deg R_p^+=-m_+-p$ for $|\xi'|>1$ and operators with symbols $A_p^-(x,\xi',\xi_n)$ with $\mathrm{deg}A_p^-=m_+-p$ with $|\xi'|>1$ and $0\leq p\leq N$ such that
\begin{align*}
A(x,D',D_n)\sparen{\sum\limits_{p=0}^NR_p^+(x,D^-,D_n)}=\sum\limits_{p=0}^NA_p^-(x,D'D_n)+T_{m_+-N-1}.
\end{align*}
The operators $T_{m_P+-N-1}$ are of lower order: 
\begin{align*}
\norm{T_{m_P+-N-1} u}_{0,s}\leq C \norm{u}_{m_+-1,s-N} \qquad \forall s.
\end{align*}
Because $A$ is a differential operator, the product
\begin{align*}
A(x,D)R_p^+(x,D',D_n)
\end{align*}
is of the following form
\begin{align*}
\sum\limits_{j=0}^mB_{jp}(x,D),
\end{align*}
where
\begin{align*}
&B_{0p}(x,\xi)=A_0(x,\xi)R_p^+(x,\xi',\xi_n), \\&
B_{jp}(x,\xi)=\sum\limits_{r=0}^{j-1}\sum\limits_{|k|=j-r}\frac{1}{k!}\frac{\partial^k A_r(x,\xi)}{\partial\xi^k}D_x^kR_p^+ +A_j(x,\xi)R_p^+, \quad j\geq 1, \\&
A(x,\xi)=A_0(x,\xi)+\sum\limits_{j=1}^mA_j(x,\xi) \qquad \mathrm{ord}A_j=m-j.
\end{align*}
If we take the product of $A$ and $\sum_{p=0}^NR_p^+$ and collect symbols of the same degree of homogeneity in $(\xi',\xi_n)$ with $|\xi'|>1$ we obtain
\begin{align*}
&A_0(x,\xi)R_0^+(x,\xi',\xi_n)=A_0^-(x,\xi',\xi_n),\\&
A_0(x,\xi)R_p^+(x,\xi',\xi_n)+T_p(x,\xi',\xi_n)=A_p^-(x,\xi',\xi_n), \qquad p\geq 1,
\end{align*}
with $T_p$ depending on $R_j^+$ for $0\leq j\leq p-1$, 
\begin{align*}
T_p=\sum\limits_{j+r=p}B_{jr},
\end{align*}
and $\mathrm{deg}T_p=m_+-p$ for $|\xi'|>1$. 

We now define the operator $\Pi'$ as 
\begin{align*}
\Pi'D(x',0,\xi',\xi_n)=\frac{1}{2\pi}\int\limits_{\gamma_+}D\,d\xi_n,
\end{align*}
with $\gamma_+$ being a contour which encloses the poles of $D$ in the upper half plane and let
\begin{align*}
b_{jk}(x',\xi')=\Pi' B_{jk}(x',0,\xi',\xi_n).
\end{align*}
Here $B_{jk}$ is the symbol of the composition of $B_j$ and $R^+\frac{\partial^{k-1}}{\partial x_n^{k-1}}$. 

The Shapiro--Lopatinskii condition in its algebraic form then says that 
\begin{align}\label{bc}
\det[b_{jk0}(x',\xi')]_{j,k=1}^{m_+}\neq 0 \qquad \forall (x',\xi'), \xi'\neq 0.
\end{align}
 Here $b_{jk0}(x',\xi')$ denotes the prinicipal part of $b_{jk}$.

 We let $A(x,D)$ be elliptic in $\overline{\Omega}$ and the Shapiro--Lopatinskii condition (\ref{bc}) be satisfied. We let 
\begin{align*}
s>\max\limits_{1\leq j\leq m_+}\sparen{m_j+1/2}, 
\end{align*}
with $m_j=\deg B_j$. Under these conditions, we then have the following theorem:
\begin{thm}\label{fredholm2}
The boundary value problem defines a Fredholm operator from $H_{s}(\Omega)$
to 
\begin{align*}
H_{s-m}\times\prod\limits_{j=1}^{m_+}H_{s-m_j-1/2}(\partial\Omega).
\end{align*}
There exists a constant $C$ such that
\begin{align*}
\norm{u}_{s, \Omega}\leq C\sparen{\norm{f}_{s-m, \Omega}+\sum\limits_{j=1}^m\norm{g_j}_{s-m_j-1/2, \partial\Omega}+\norm{u}_{s-1, \Omega}}.
\end{align*}
\end{thm}
A similar theorem can be found in the earlier work \cite{gerd}, but we use the formulation of \cite{eskin} for the ease of computation of the symbol classes. 

\section{Stability in Dimension 2}\label{stab2d}
We prove the following analogue of Theorem 3 in \cite{SAMRE}. 
\begin{thm}\label{conditional2}
Let $\Omega \subset \mathbb{R}^2$ a smooth and bounded domain, $(u_1,p_1)$, $(u_2,p_2)$ be two solutions to equation \eqref{stokesdata} with coefficients $\mu_1, \mu_2 \in \mathcal{C}^4(\bar{\Omega})$, respectively. Assume that $\mu_1=\mu_2$ on $\partial\Omega$ and
\begin{equation} \label{cond-2d}
 |\nabla^s u_1(x)| \neq 0, \quad x\in \bar{\Omega}.
\end{equation}
Let $\epsilon\in (0,1)$. Then there exists a finite dimensional subspace $K \subset L^2(\Omega)$ and a constant $C>0$, depending on $\|\mu_2\|_{\mathcal{C}^{5/2+\epsilon}(\bar{\Omega})}$, such that 
\begin{align}\label{mainest2}
\norm{\mu_1-\mu_2}_{1/2+\epsilon,\Omega}\leq C\norm{u_1-u_2}_{1/2+\epsilon,\Omega},
\end{align}
provided  $\mu_1-\mu_2 \perp K$.
\end{thm}
 In \cite{SAMRE} under the same hypotheses we proved
\begin{align*}
\norm{\mu_1-\mu_2}_{4,\Omega}\leq C\norm{u_1-u_2}_{5,\Omega}
\end{align*}
The difference between this theorem and Theorem 4 in \cite{SAMRE} is the norms on the estimates in \ref{mainest2}, and the shorter proof. The reduction in regularity is more suitable for numerical experiments as discussed in Section \ref{secfinal}. 
\begin{proof}
For the proof, we may assume without loss of generality that $\partial_{x_1}u_1^1(x) \neq 0$. We apply the operator $(\partial_{x_1},-\partial_{x_2})$ to \eqref{stokesdata} in order to eliminate the pressure term. For $\mu=\mu_1-\mu_2$, we derive
\begin{align}\label{mueq}
(\partial_{x_1},-\partial_{x_2})\cdot (2\nabla\cdot\mu \nabla^su_1)=g.
\end{align}
By definition, $g$ satisfies
\begin{align*}
\norm{g}_{l, \Omega}\leq C\norm{u_1-u_2}_{l+2, \Omega}
\end{align*}
for all $l\leq 0$. The principal symbol of the linear operator in \eqref{mueq} is given by
\begin{align*}
A_0(x,\xi)=2|\xi|^2\partial_{x_1}u_1^1(x).
\end{align*}
In order to verify the Shapiro-Lopatinskii condition (\ref{bc}), we compute the zeros of the polynomial in terms of $\xi_1$. Provided $\partial_{x_1}u(x)\neq 0$, the root in the upper half plane can be computed as
\begin{align*}
\xi_1=\pm i\xi_2.
\end{align*}
It follows that for $j=0,1$
\begin{align*}
b_{j0}=\frac{1}{2\pi i}\int\limits_{\gamma_+}\frac{1}{z-i}\,dz \neq 0.
\end{align*}
We then apply Theorem \ref{fredholm2}.
\end{proof}

\section{Stability in Dimension 3}\label{stab3d}
We prove the following theorem, which is the analogue of Theorem 3 in \cite{SAMRE}.
\begin{thm}\label{conditional}
Let $(u_1,p_1)$ and $(\tilde{u}_1,\tilde{p}_1)$ be solutions to (\ref{stokesdata}) with different boundary conditions.  In other words we set $F(x)=F_1(x)$ and $F(x)=\tilde{F}_1(x)$ in (\ref{stokesdata}) for the respective solutions but they share $\mu=\mu_1$. We assume that there exists a positive constant $C$  independent of $(x,\xi)\in T^*\bar{\Omega}$, where $T^* \bar{\Omega}$ denotes the cotangent space, such that
\begin{align}
%C|\xi|^2\geq |(\nabla^su_1(x)\xi)\times\xi|+|(\nabla^s\tilde{u}_1(x)\xi)\times\xi|\geq \frac{1}{C}|\xi|^2.
C|\xi|^4\geq |(\nabla^su_1(x)\xi)\times\xi|^2+|(\nabla^s\tilde{u}_1(x)\xi)\times\xi|^2\geq \frac{1}{C}|\xi|^4.
\label{stabilizerssecond}
\end{align}
Let $(u_2,p_2)$ and $(\tilde{u}_2,\tilde{p}_2)$ be solutions to the Stokes system (\ref{stokesdata}) with $\mu=\mu_2$ and $F(x)=F_1(x)$ and $F(x)=\tilde{F}_1(x)$, respectively. Let $\epsilon\in (0,1)$. Assume that $\mu_1, \mu_2 \in \mathcal{C}^{7/2+\epsilon}(\bar{\Omega})$ and $\mu_1=\mu_2$ on $\partial\Omega$.  Then
there exist a constant $C$, depending on $\|\mu_2\|_{\mathcal{C}^{7/2+\epsilon}(\bar{\Omega})}$ and a finite dimensional subspace $K$ of $H^{1/2+\epsilon}(\Omega)$ such that
\begin{align}\label{mainest}
\norm{\mu_1-\mu_2}_{1/2+\epsilon,\Omega}\leq C \sparen{\norm{u_1-u_2}_{3/2+\epsilon, \Omega}+\norm{\tilde{u}_1-\tilde{u}_2}_{3/2+\epsilon,\Omega}},
\end{align}
provided that $(\mu_1-\mu_2) \perp K $.
\end{thm}
The main difference from Theorem 3 in \cite{SAMRE} is the reduction of regularity on the norms. However there is a slightly stronger assumption on the symbol classes, in \cite{SAMRE} the upper bound was not necessary. 

We begin as previously by eliminating the pressure terms from the Stokes systems. We consider the equations for $i=1,2$ which are
\begin{align*}
\nabla\cdot\mu_i\nabla^su_i+\omega^2u_i+\nabla p_i=0
\end{align*}
When we take the cross product of both sides this yields the equation
\begin{align*}
\nabla\times\nabla\cdot \mu_i\nabla^s u_i+\omega^2\nabla\times u_i=0
\end{align*}
If we set $\mu=\mu_1-\mu_2$, $w=u_1-u_2$ and subtract the first equation from the second equation, we obtain
\begin{align}\label{ida}
A_{u_1} \mu = \nabla\times\nabla\cdot \mu \nabla^s u_1=g
\end{align}
with 
\begin{align*}
g=-\nabla\times[\nabla\cdot\mu_2\nabla^s w]-\omega^2\nabla\times w.
\end{align*}
Clearly we then have that there is a constant $C$ which depends on $\norm{\mu_2}_{C^{2+l}(\overline{\Omega})}$ norm such that
\begin{align*}
\norm{g}_{l, \Omega}\leq C\norm{w}_{l+3,\Omega}
\end{align*}
for all $l$. We view the identity \eqref{ida} as an over-determined second-order partial differential equation with principal symbol 
\begin{align*}
(\nabla^su_1(x)\xi)\times\xi
\end{align*}
which is not elliptic. We augment the operator with a second set of measurements: 
\begin{align*}
A_{u_2} \mu = \nabla\times\nabla\cdot \mu \nabla^s u_2=\tilde{g}.
\end{align*}

We set
\begin{align*}
A_*=(A_{u_1},A_{\tilde{u}_1}) \qquad B_*=(\mbox{Trace on } \partial \Omega,\mbox{Trace on } \partial \Omega),
\end{align*}
and we analyse the new system
\begin{align}\label{system}
\mathcal{A}_*[\mu]=(A_*[\mu], B_*[\mu])=(g,\tilde{g}, 0,0).
\end{align}
The condition (\ref{stabilizerssecond}) ensures that the system is now injectively elliptic, whence $\mathcal{A}^*_*\mathcal{A}_*$ is elliptic and satisfies the assumptions of Section \ref{fredsection}. If $(\mathcal{A}^*_*\mathcal{A}_*)^{\sim}$ is a parametrix of $\mathcal{A}^*_*\mathcal{A}_*$, a left parametrix of $\mathcal{A}_*$ is given by $(\mathcal{A}^*_*\mathcal{A}_*)^{\sim} \mathcal{A}^*_*$. 

Therefore, if the augmented system defines an elliptic boundary value problem, we conclude:
$$\|\mu\|_{s,\Omega}\leq C(\|g\|_{s-2, \Omega} + \|\tilde{g}\|_{s-2, \Omega}) \leq C (\|u_1-u_2\|_{s+1, \Omega}+\|\tilde{u}_1-\tilde{u}_2\|_{s+1, \Omega})\ .$$

%(((We also assumed an upper bound, which could be replaced by some other conditions.)))
%Let $r=\sqrt{\xi^2_2+\xi_3^2}$, then under the hypothesis (\ref{stabilizerssecond}), we know $|A_0(x,\xi)|$ has two imaginary roots in the intervals 
%$[ir/\sqrt{C},i\sqrt{C}r]$ and $[-i\sqrt{C}r,-ir/\sqrt{C}]$ respectively. This result follows from the intermediate value theorem since the roots of the left hand side and right hand side of (\ref{stabilizerssecond}) interlace the roots of $|A_0(x,\xi)|$. 
To show the Shapiro--Lopatinskii condition \eqref{bc}, we argue as follows. According to Hypothesis (\ref{stabilizerssecond}), the polynomial $|(\nabla^su_1(x)\xi)\times\xi|^2+|(\nabla^s\tilde{u}_1(x)\xi)\times\xi|^2$ is elliptic and therefore, if $r=\sqrt{\xi^2_2+\xi_3^2} \neq 0$, two of its roots each lie in the upper resp.~lower half-plane. As the roots  are homogeneous in $r$, it is easy to show that $\det\sparen{b_{jk0}(x',\xi')}^{m_+}_{j,k}\neq 0$. We then apply Theorem \ref{fredholm2}. 

\section{Reduction of Regularity in Weighted Sobolev Classes}
To further reduce the order in the stability estimates of Theorems \ref{conditional2} and \ref{conditional}, we recall the Fredholm properties of a Shapiro-Lopatinskii elliptic boundary value problem, $\mathcal{A}=(A,B)$ with $\mathrm{ord}(A)=m$. See \cite{lionsmag} for more detailed results.
 
Fix a boundary defining function $\rho\in C^{\infty}(\Omega)$. For $s\in \mathbb{N}_0$ we let
\begin{align*}
\Theta^s=\{u: \rho^{|\alpha|}D^{\alpha}u\in L^2 \,\,\forall \alpha, |\alpha|\leq s\}.
\end{align*}
We set $\Theta^{-s}=(\Theta)^*$. Elements in $\Theta^{-s}$ are of the form
\begin{align*}
f=\sum\limits_{|\alpha|\leq s} D^{\alpha}(\rho^{|\alpha|}f_{\alpha})\,\,\, \textrm{for some}\,\, f_{\alpha}\in L^2. 
\end{align*}
Let
\begin{align*}
D_A^0(\Omega)=\{ u\in L^2: Au\in \Theta^{-s}(\Omega)\}
\end{align*}
endowed with the graph norm of $A$. The basic Fredholm theorem then reads:
\begin{thm}
$\mathcal{A}$ defines a Fredholm operator from $D_{A}^0(\Omega)$ to 
\begin{align*}
\Theta^{-m}(\Omega)\times \prod\limits_{j=0}^{m_+-1}H^{s-m_j-1/2}(\Gamma).
\end{align*}
\end{thm}
This allows us to conclude in dimension 2:
\begin{cor}
Under the hypotheses of Theorem \ref{conditional2}, there exists a finite dimensional subspace $K \subset L^2(\Omega)$ and a constant $C>0$ depending on $\mu_1$ and $\Omega$ such that 
\begin{align}\label{mainest2}
\norm{\mu_1-\mu_2}_{L^2(\Omega)}\leq C\norm{\rho^{-2}(u_1-u_2)}_{L^2(\Omega)}
\end{align}
provided  $\mu_1-\mu_2 \perp K$.
\end{cor}
In dimension 3 we have:
\begin{cor}
Under the hypotheses of Theorem \ref{conditional}, and there exists a constant $C$, depending on $\mu_1,\tilde{\mu}_1$ and $\Omega$ and a finite dimensional subspace $K$ of $L^2(\Omega)$ such that
\begin{align*}
\norm{\mu_1-\mu_2}_{L^2(\Omega)}\leq C \sparen{\norm{\rho^{-2}(u_1-u_2)}_{H^1(\Omega)}+\norm{\rho^{-2}(\tilde{u}_1-\tilde{u}_2)}_{H^1(\Omega)}},
\end{align*}
provided that $(\mu_1-\mu_2) \perp K$.
\end{cor}

\section{Local Injectivity Results}\label{injectivity}
Recall that we say the principal symbol is non-characterstic at the origin if $p(x,\xi)$ is such that $p(0,\xi)\neq 0$ with $\xi=(0,0, . .,0,1)$. We set $\tau=\xi_{n}$. We assume the following three hypotheses, given some $\epsilon>0$
\begin{enumerate}
\item The symbol $p(x,\xi',\tau)=0$ has at most simple real zeros and at most double complex zeros in terms of $\tau$. 
\item If $\tau_1,\tau_2$ are distinct zeros of $p(x,\xi',\tau)=0$ then this implies $|\tau_1-\tau_2|\geq \epsilon$. 
\item For any non-real zero $\tau$ of $p(x,\xi',\tau)$, we have $|\Im \tau|\geq \epsilon$.
\end{enumerate}

We recall the following uniqueness result, based on a refined Carleman estimate \cite{nirenberg}. 

\begin{thm}
Assume that the plane $x_{n}=0$ is non-characteristic at the origin and the Hypotheses (1-3) above hold. Let $l\geq 0$ and $w \in H^l(\mathbb{R}^d)$ a solution of $Pw=0$ in a neighborhood of the origin which vanishes identically for $x_{n}<0$. Then $w\equiv 0$ in a full neighborhood of the origin. 
\end{thm}

We prove:
\begin{cor}\label{cor2d}
We assume in $d=2$ that $\partial_{x_1}^2u(x)\neq 0$. If $\mu=0$ in some neighborhood of the hyper-plane $x_n=0$, and $u_1-u_2=0$, then $\mu\equiv 0$ in a neighborhood of the origin. 
\end{cor}
\begin{proof}
We recall that the principal symbol is
\begin{align*}
2|\xi|^2\partial_{x_1}^2u^1.
\end{align*}
Setting $\tau=\xi_1$ we are looking for roots of the equation $\tau^2=-\xi_2^2$. These solutions $\tau=\pm i\xi_2$ clearly satisfy the necessary hypotheses. We must also check that the polynomial associated to $A_0(x,\xi)$ is non-characteristic. The hypothesis is satisfied under the assumption $\partial_{x_1}^2u(0,x')\neq0$, which we already assumed in order that the operator be elliptic.
\end{proof} 
We also prove:
\begin{cor}\label{cor3d}
We assume in $d=3$ that the condition (\ref{stabilizerssecond}) holds. If $\mu=0$ in some neighborhood of the hyper-plane $x_n=0$, and $\tilde{u}_1-\tilde{u}_2=0$ and $u_1-u_2=0$, then $\mu\equiv 0$ in a neighborhood of the origin. 
\end{cor}
\begin{proof}
We recall the principal symbol
\begin{align*}
|A_0(x,\xi)|=|(\nabla^su_1\xi)\times \xi)|+|(\nabla^s\tilde{u}_1\xi)\times\xi|.
\end{align*}
 For $A_0(x,\xi)$ to be noncharacteristic, we compute the symbol. We need only that one of the off-diagonal terms in the matrices defining the symmetric gradients $\nabla^su_1$ and $\nabla^s\tilde{u}_1$ be nonzero at the origin. The roots of the polynomial satisfying conditions (i-iii) are determined by the condition (\ref{stabilizerssecond}), as discussed in Section \ref{stab3d}.
 \end{proof}
 
\section{Examples Illustrating the Boundary Conditions}\label{example}
We would like to show that there are natural means to enforce the somewhat unusual looking hypotheses on the symbols given by \eqref{stabilizerssecond} in $d=3$. 
\begin{lem}
It is possible to chose initial data $F$ and $\tilde{F}$ to the system \eqref{stokesdata} such that the condition \eqref{stabilizerssecond} holds a.e.
\end{lem}
\begin{proof}
We let $u_1$ be a solution to the system \eqref{stokesdata}. Here we note that $\nabla^su_1$ satisfies a closely related elliptic problem 
\begin{align}\label{nabladata}
\omega^2\nabla^su_1-\nabla^s(\nabla\cdot\mu_1\nabla^s u_1)=\nabla^s\nabla p
\end{align}
with corresponding boundary conditions. The boundary data $\nabla^s u_1|_{\partial\Omega}$ is determined by $F$ through a Dirichlet-to-Neumann map. We chose $F$ to be a random field on $\partial\Omega$ such that $\nabla^su_1|_{\partial\Omega}$ is a random field of Wigner matrices. In other words, the entries of $\nabla^s u_1|_{\partial\Omega}$ are distributed according to a Gaussian law and $\mathrm{tr} \nabla^s u_1|_{\partial\Omega}=0$. Because $u_1$ is a solution to \eqref{stokesdata}, it is a given that $\mathrm{tr}\nabla^s u_1=0$ in $\Omega$. Because the solution operator $R$ which is associated to 
\eqref{nabladata} is a linear pseudo-differential boundary operator, $\nabla^s u(x)=R(\nabla^su|_{\partial\Omega})$ still obeys a Gaussian law for every $x\in \Omega$. 
It follows as a result of \cite{terry} that 
\begin{align*}
P(v: \exists \lambda\in \mathbb{R}^+, Av=\lambda v, Bv=\lambda v)
\end{align*}
for $A,B$ with $A\neq B$ Wigner matrices is small. In other words, $\xi$ is an eigenvector of $\nabla^su$ and $\nabla^s \tilde{u}$ with small probability provided $\nabla^su_1|_{\partial\Omega}$ and $\nabla^s\tilde{u_1}_{\partial\Omega}$ are chosen as Wigner matrices. This necessarily implies the condition \ref{stabilizerssecond} with a high probability. The natural condition in dimension $d=2$ is that $\nabla^s u>0$, which can be enforced by similar means. 
\end{proof}
\section{Stability of the Landweber Iteration Scheme}\label{secfinal}
We consider as in \cite{SAMRE} the reconstruction of the true shear modulus distribution $\mu_{{\rm tr}}$. The treatment follows the book \cite{ammari_book_elas} closely. We want to reconstruct the $\mu$ from $u_{m}$ which is the measured displacement field. We introduce the functional 
\begin{align*}
\mathcal{J}[\mu]=\frac{1}{2}\int\limits_{\Omega}\sabs{u-u_{m}}^2\,dx, 
\end{align*}
where we have that $u$ is the solution to \eqref{stokesdata} which minimizes $\mathcal{J}[\mu]$ when $\mu$ is varied. Fixing $\mu$ we consider a solution $v$ to
\begin{align*}
\left\{
\begin{array}{lr}
2\nabla\cdot\mu\nabla^s v+\omega^2 v +\nabla p =
(\overline{u-u_m}) \quad \mathrm{in} \quad \Omega ,\\ \\
\nabla\cdot v =0 \quad \mathrm{in} \quad \Omega ,\\ \\
v=0 \quad \mathrm{on}\quad  \partial\Omega ,\\ \\
\int\limits_{\Omega}p \,dx=0 .
\end{array}
\right.
\end{align*}
We then compute the Fr\'echet derivative $D\mathcal{J}[\mu]$ of $\mathcal{J}$  as
\begin{align*}
<D\mathcal{J}[\mu],\delta\mu>=\int\limits_{\Omega}\delta\mu\nabla^sv : \nabla^su\,dx.
\end{align*}
This identifies the functional $D\mathcal{J}[\mu]$ with$\nabla^sv:\nabla^su.$

With the gradient descent method, the numerical minimization of $\mathcal{J}$ amounts to the following. After an initial guess $\mu_0$, we update it with the following scheme:
\begin{align}\label{bestguess}
\mu_{n+1}(x)=\mu_n(x)-\sigma D\mathcal{J}[\mu_n](x) \qquad x\in \Omega,\,\, n\geq 0,
\end{align}
with $\sigma$ being the step size. This procedure is outlined in \cite{ammari_book_elas}.

Following \cite{laure}, the mapping $\mathcal{F}$
\begin{align*}
\mathcal{F}: \mu\mapsto u,
\end{align*}
is such that 
\begin{align*}
D\mathcal{J}[\mu]=(D\mathcal{F}[\mu])^*(\mathcal{F}[\mu]-\mathcal{F}[\mu_{\rm tr}]),
\end{align*}
where the superscript $*$ denotes the adjoint. 

The resulting optimal control scheme (\ref{bestguess}) is controlled by a 
Landweber iteration scheme given by
\begin{align*}
\mu_{n+1}(x)=\mu_n(x)-\sigma(D\mathcal{F}[\mu])^*(\mathcal{F}[\mu]-\mathcal{F}[\mu_{\rm tr}])(x) \qquad x\in \Omega,\,\, n\geq 0 .
\end{align*}

From \cite[Appendix A]{laure}, the following convergence result in $L^2(\Omega)$ for the Landweber (or equivalently the optimal control) scheme holds.
\begin{thm}
Let $d=2$. 
Assume that the assumptions of Theorem \ref{conditional2} are satisfied and $K$ is trivial.  If, for  sufficiently small $\epsilon_0$ and some $\epsilon\in (0,1)$,
\begin{align*}
\norm{\mu_0 -\mu_{\rm tr}}_{H^{1/2+\epsilon}(\Omega)}<\epsilon_0 ,
\end{align*}
then
\begin{align*}
\norm{\mu_n-\mu_{\rm tr}}_{H^{1/2+\epsilon}(\Omega)} \rightarrow 0 \qquad \mathrm{as} \, \, \,  n\rightarrow \infty.
\end{align*}
\end{thm}
The modification between this theorem and the result in \cite{SAMRE} is the reduction of the necessary regularity to practically relevant values. The discrepancy functional in $d=3$ must be modified to include both sets of data as in \cite{SAMRE}. We can then conclude that: 

\begin{thm}
Let $d=3$. Assume that the assumptions of Theorem \ref{conditional} are satisfied and that $K$ is trivial. 
If, for sufficiently small $\epsilon_0$ and some $\epsilon\in (0,1)$, 
\begin{align*}
\norm{\mu_0-\mu_*}_{H^{1/2+\epsilon}(\Omega)}<\epsilon_0,
\end{align*}
then we have 
\begin{align*}
\norm{\mu_n-\mu_{\rm tr}}_{H^{1/2+\epsilon}(\Omega)} \rightarrow 0 \qquad \mathrm{as} \, \, \,  n\rightarrow \infty.
\end{align*}
\end{thm}

\section*{Acknowledgments}
A.~W.~acknowledges support by EPSRC grant EP/L01937X/1 and ERC Advanced Grant MULTIMOD 26718. H.~G.~is supported by a PECRE award of the Scottish Funding Council and ERC Advanced Grant HARG 268105. Both authors thank Gerd Grubb for useful comments. 

%\renewcommand\refname{\large References}
%\bibliography{elastography.bib}

\end{document}